\documentclass[10pt]{article}
\usepackage{amsmath, amsthm, amssymb,todonotes, xcolor,stackrel}
\usepackage{graphicx,authblk,mathtools,relsize}
\usepackage[all]{xy}
\xyoption{all}
\newtheorem{corollary}{Corollary}
\newtheorem{proposition}{Proposition}
\newtheorem{lemma}{Lemma}
\newtheorem{theorem}{Theorem}
\newtheorem{claim}{Claim}

\newtheorem*{theorem*}{Theorem}

\newtheorem{observation}{Observation}

\makeatletter
\DeclareRobustCommand*\cal{\@fontswitch\relax\mathcal}
\makeatother

\theoremstyle{definition}
\newtheorem{definition}{Definition}
\newtheorem{example}{Example}

\newcommand\wt{\widetilde}
\newcommand\ov{\overline}
\newcommand\inj{\rightarrowtail}
\newcommand\surj{\twoheadrightarrow}

\mathchardef\mhyphen="2D

\newcommand{\inv}{^{\raisebox{.2ex}{$\scriptscriptstyle-1$}}} % Superscript -1
\DeclareMathOperator{\im}{im} % image of Morphism

\newcommand{\lra}[1]{\overset{ #1 }{\longrightarrow}} % Arrow with morphism name
\newcommand{\ul}{\underline}

\newcommand{\leftrarrows}{\mathrel{\raise.75ex\hbox{\oalign{%
  $\scriptstyle\leftarrow$\cr
  \vrule width0pt height.5ex$\hfil\scriptstyle\relbar$\cr}}}}
\newcommand{\lrightarrows}{\mathrel{\raise.75ex\hbox{\oalign{%
  $\scriptstyle\relbar$\hfil\cr
  $\scriptstyle\vrule width0pt height.5ex\smash\rightarrow$\cr}}}}
\newcommand{\Rrelbar}{\mathrel{\raise.75ex\hbox{\oalign{%
  $\scriptstyle\relbar$\cr
  \vrule width0pt height.5ex$\scriptstyle\relbar$}}}}

\makeatletter
\newcommand{\subjclass}[2][2010]{%
  \let\@oldtitle\@title%
  \gdef\@title{\@oldtitle\footnotetext{#1 \emph{Mathematics subject classification.} #2}}%
}
\newcommand{\keywords}[1]{%
  \let\@@oldtitle\@title%
  \gdef\@title{\@@oldtitle\footnotetext{\emph{Key words and phrases.} #1.}}%
}
\makeatother

\begin{document}

\title{On the structure of modules indexed by small categories}
\author{Crichton Ogle\thanks{corresponding author: ogle.1@osu.edu\\ MSC primary 16-XX; secondary 55-XX}, Sami Sultan} 
\affil{Dept. of Mathematics, OSU}
%\subjclass[2][2010]{MSC primary 16-XX; secondary 55-XX}
\date{}
\maketitle

\begin{abstract} Given a small catgory $\cal C$ and $\cal C$-diagram of vector spaces $M$ over a field $k$ (referred to as a $\cal C$-module), we show that $M$ admits a quasi-tame $\cal C$-module cover ${\cal QTC}(M)$, determined by its local structure ${\cal F}(M)$ when the local structure is stable (which holds whenever both $\cal C$ and $k$ are finite). ${\cal QTC}(M)$ is a finite direct sum of quasi-blocks determined by the associated graded local structure ${\cal F}(M)_*$ of $M$, and for all $M$ there is a natural isomorphism ${\cal F}({\cal QTC}(M))_*\cong{\cal F}(M)_*$ of associated graded local structures. When the local structure is stable, it determines an effectively computable non-negative integer invariant $e(M)$ - the excess of $M$ - that quantifies the degree to which the local structure fails to be in general position at each object of $\cal C$. When $M$ admits an inner product there is a surjection of $\cal C$-modules ${\cal QTC}(M)\surj M$ inducing the above isomorphism of associated graded local structures, which is an isomorphism iff $e(M)=0$. In the very special case $\cal C$ is the categorical realization of a finite totally ordered set, these methods recover the classical result that any finite one-dimensional persistence module decomposes as a direct sum of interval submodules in a manner unique up to re-ordering.
\end{abstract}
%\vskip.3in

%\tableofcontents

\section{Introduction} It is a consequence of Gabriel's structure theorem for quiver representations of type $A_n$ \cite{pg} that any finite 1-dimensional persistence module over a field $k$ decomposes into a direct sum of indecomposable interval modules; moreover, the decomposition is unique up to reordering. The barcodes (with multiplicity) associated to the original module correspond to these interval submodules, which are indexed by the set of connected subgraphs of the finite directed graph associated to the finite, totally ordered indexing set of the module. Modules that decompose in such a fashion are conventionally referred to as {\it tame}.
\vskip.2in
A central problem in topological data analysis has been to determine what, if anything, holds true for more complex types of modules, for example $n$-dimensional persistence modules \cite{cs, cz}, given the fact that these more general types of modules are almost never tame.
\vskip.2in

This paper provides an answer to that question; what follows is a summary of the contents. 
\vskip.2in

In section 2.1, we define the framework in which we will be working. We consider modules indexed by an arbitrary small category $\cal C$ (referred to as $\cal C$-modules in this paper) equipped with i) no additional structure, and ii) an {\it inner product}. The structural properties we later establish for such modules in section 2.2 are based on the fundamental notion of  a {\it multi-flag} of a vector space $V$, defined as a collection of subspaces of $V$ that includes $\{{\bf 0}\}, V$, and is closed under intersection. Equally important is the notion of {\it general position} for such an array. Using terminology made precise in section 2.3,  a $\cal C$-module $M$ determines a functor ${\cal F}:{\cal C}\to (multi\mhyphen flags/k)$ which associates to each $x\in obj({\cal C})$ a multi-flag ${\cal F}(M)(x)$ of the vector space $M(x)$, referred to as the {\it local structure} of $M$ at $x$. The local structure is naturally the direct limit of a directed system of recursively defined multi-flags $\{{\cal F}_n(M), \iota_n\}$, and is {\it stable} when this directed system stabilizes at a finite stage (with the limit written as ${\cal F}(M)$). This structure can be refined by the existence of an inner product on $M$ (section 2.4). Section 2.5 describes the behavior of the local structure with respect to sums and tensors of $\cal C$-modules, and for tensor products proves an algebraic K\"unneth Theorem for local structures (used later on in section 4).
\vskip.2in

Assuming a $\cal C$-module $M$ has stable local structure, the associated graded ${\cal F}_*(M)$ defines a $\cal C$-module ${\cal QTC}(M)$, referred to as the {\it quasi-tame cover} of $M$, constructed in section 3.1. Under suitable conditons on $M$ the quasi-tame cover is \emph{b-tame} (section 3.2), meaning it is a direct sum of {\it blocks} (generalizing the interval decomposition of 1-dimensional persistence modules). For all ${\cal C}$-modules $M$ there is a canonical isomorphism of associated gradeds
\begin{equation}\label{eqn:eq1}
{\cal F}({\cal QTC}(M))_*\xrightarrow{\cong} {\cal F}(M)_*
\end{equation}
The module ${\cal QTC}(M)$ represents the closest approximation to $M$ by a quasi-tame module (one which decomposes as a finite direct sum of quasi-blocks), and equals $M$ precisely when $M$ itself is quasi-tame. In section 3.3 we show that when $M$ has stable local structure and admits an inner product (is an $\cal IPC$-module), the map of associated gradeds in (1) is induced by a surjection of $\cal C$-modules
\begin{equation}\label{eqn:eq2}
{\cal QTC}(M)\surj M
\end{equation}
Morover, for $\cal IPC$-modules the following statements are equivalent
\begin{itemize}
\item The surjection in (\ref{eqn:eq2}) is an isomorphism;
\item $M$ is quasi-tame;
\item the local structure of $M$ is in general position at each object of $\cal C$;
\item the excess $e(M)$ of $M$ - a computable non-negative integral invariant of $M$ - is zero.
\end{itemize}
In section 3.4 we show that when $M$ is a 1-dim.~persistence module (i.e., when it is a $\cal C$-module where $\cal C$ is the categorical realization of a finite totally ordered set), there is an isomorphism ${\cal QTC}(M)\cong M$ which recovers the classical decomposition for such modules. In section 3.5 we show that all finite $n$-dimensional persistence modules have stable local structure.
\vskip.2in

Finally in section 4 we prove two results for topologically based $\cal C$-modules - those resulting via applying $H_*(_-)$ to a $\cal C$-diagram of spaces. The second of the two is a K\"unneth Theorem for such modules, derived via application of the algebraic K\"unneth Theorem of section 2.5 to this topologically based situation.
\vskip.2in

We would like to thank Dan Burghelea and Fedor Manin for their helpful comments on earlier drafts of this work, as well as Bill Dwyer for his contribution to the proof of the cofibrancy replacement result presented in section 4. We are indebted to Jakob Hansen for valuable feedback and discussions at various later stages of this work. This paper is part of a larger joint project with Sanjeevi Krishnan.
\vskip.5in

%%%%%%%%%%%%%%%%%%%%%%%%%%%%%%%%%%%%%%%%%%%%%%%%%
%%%%%%%%%%%%%%%%%%%%%%%%%%%%%%%%%%%%%%%%%%%%%%%%%

\section{$\cal C$-modules} 

\subsection{Preliminaries} Throughout we work over a fixed field $k$. Let $(vect/k)$ denote the category of finite-dimensional vector spaces over $k$, and linear homomorphisms between such. Additionally we denote by $(vect/k)_*$ the category of finite-dimensional vector spaces over $k$, and linear homomorphisms enriched over pointed sets with the base point corresponding the the zero linear map. 

\begin{definition} Given a small category $\cal C$, a \emph{$\cal C$-module} is a covariant functor $M:{\cal C}\to (vect/k)$. Similarly, given a small category ${\cal C}_*$ enriched over pointed sets, a \emph{basepointed $\cal C$-module} is a basepointed functor $M:{\cal C}_*\to (vect/k)_*$.
\end{definition}

 [Note: This basepointed structure above ensures that $Hom(M(x), M(y))$ always contains the zero linear map. The only place where this zero map will be needed is in showing that  the associated graded of the local structure for basepointed $\cal C$-modules distributes over finite sums of modules.]
 \vskip.2in
 
The category $({\cal C}\mhyphen mod)$ of $\cal C$-modules then has these functors as objects, with morphisms represented in the obvious way by natural transformations.  All functorial constructions on vector spaces extend to the objects of $({\cal C}\mhyphen mod)$ by objectwise application. In particular, one has the appropriate notions of
\begin{itemize}
\item monomorphisms, epimorphisms, short and long-exact sequences;
\item kernel and cokernel;
\item direct sums, Hom-spaces, tensor products;
\item linear combinations of morphisms.
\end{itemize}

With these constructs $({\cal C}\mhyphen mod)$ is an abelian category. 
\vskip.3in

%%%%%%%%%%%%%%%%%%%%%%%%%%%%%%%%%%%%%%%%%%%%%%%%%

\subsection{Multi-flags and general position} Recall that a {\it flag} in a vector space $V$ consists of a finite sequence of proper inclusions beginning at $\{{\bf 0}\}$ and ending at $V$:
\[
\underline{W} := \{W_i\}_{0\le i\le n} = \left\{\{{\bf 0}\} = W_0\subset W_1\subset W_2\subset\dots\subset W_m = V\right\}
\]
We will relax this structure in two different ways. 
\vskip.2in

\begin{definition} A \emph{semi-flag} in $V$ is a sequence of not necessarily proper inclusions\newline $\left\{\{{\bf 0}\} = W_0\subseteq W_1\subseteq W_2\subseteq\dots\subseteq W_m = V\right\}$. More generally, a \emph{multi-flag} in $V$ is a collection\newline ${\cal F} = \{W_\alpha\subseteq V\}$ of subspaces of $V$ containing $\{{\bf 0}\}, V$, partially ordered by inclusion, and closed under intersection. It is not required to be finite.
\end{definition}
\vskip.2in

Given an element $W\subseteq V$ of a multi-flag $\cal F$ associated to $V$, let $S_{{\cal F}}(W) := \{U\in {\cal F}\ |\ U\subsetneq W\}$ be the set of elements of $\cal F$ that are proper subsets of $W$, and set
\[
SS_{{\cal F}}(W) := \left(\displaystyle\sum_{U\in S_{{\cal F}}(W)} U\right)\cup \{{\bf 0}\}
\]
For each $W\in {\cal F}$, $SS_{{\cal F}}(W)$ is a subspace of $W$. The set of pairs of subspaces
\[
{\cal F}^p := \{(W, SS_{\cal F}(W))\}_{W\in {\cal F}}
\]
is canonically isomorphic to $\cal F$ via the identification
\[
{\cal F}\ni W \leftrightarrow (W,SS_{\cal F}(W))\in {\cal F}^p
\]
We need to consider different possible ways a subquotient might be represented by a pair in ${\cal F}^p$. To that end, note that if $W',W\in {\cal F}$ with $W'\subseteq W$, then $SS_{\cal F}(W')\subseteq SS_{\cal F}(W)$ and there is an inclusion of pairs $\iota: (W',SS_{\cal F}(W'))\hookrightarrow (W, SS_{\cal F}(W))$. The inclusion $\iota$ is a {\it q-isomorphism} if it induces an isomorphism of quotients
\[
\ov{\iota}:W_{\cal F} := W/SS_{\cal F}(W)\xrightarrow{\cong} W'/SS_{\cal F}(W') = W'_{\cal F}
\]
Finally we say that two elements $(W,SS_{\cal F}(W)), (W', SS_{\cal F}(W'))$ of ${\cal F}^p$ are {\it q-equivalent} if they are connected by a zig-zag sequence of q-isomorphisms. This defines an equivalence relation ``$\underset{q}{\sim}$" on ${\cal F}^p\cong {\cal F}$ and the {\it associated graded} ${\cal F}_*$ of $\cal F$ is defined as 
\begin{equation}\label{eqn:assoc}
{\cal F}_* := {\cal F}^p/{\underset{q}{\sim}}
\end{equation}
An element of ${\cal F}_*$ will be typically written as $[W_{\cal F}]$. Thus $[W_{\cal F}] = [W'_{\cal F}]$ iff $(W,SS_{\cal F}(W))\underset{q}{\sim} (W',SS_{\cal F}(W'))$. We say $[W_{\cal F}]\in\cal F$ is {\it without multiplicity} iff it  has only one representative (in other words, if $(W,SS_{\cal F}(W))$ is only q-equivalent to itself). 

\begin{proposition}\label{prop:multi} If $[\{{\bf 0}\}]\ne [W_{\cal F}]\in {\cal F}_*$, then $[W_{\cal F}]$ is without multiplicity.
\end{proposition}

\begin{proof} If $(W,SS_{\cal F}(W))\underset{q}{\sim} (W',SS_{\cal F}(W'))$ and $W\ne W'$, then in the zig-zag sequence of q-isomorphisms connecting $(W,SS_{\cal F}(W))$ and $(W',SS_{\cal F}(W'))$ there must be a q-isomorphism $(W_1,SS_{\cal F}(W_1))\hookrightarrow (W_2,SS_{\cal F}(W_2))$ where $W_1\hookrightarrow W_2$ is a proper inclusion. As $W_1,W_2\in {\cal F}$, this implies $W_1\in S_{F}(W_2)$, and so the induced isomorphism on quotients $W_1/SS_{\cal F}(W_1)\xrightarrow{\cong} W_2/SS_{\cal F}(W_2)$ must be the zero map. Hence 
\[
[W_{\cal F}] = [(W_1)_{\cal F}] = [W'_{\cal F}] = [\{{\bf 0}\}]
\]
\end{proof}
\vskip.2in

\begin{observation} For any multi-flag $\cal F$ of $V$, $\displaystyle\sum_{[W_{\cal F}]\in{{\cal F}_*}} dim(W_{\cal F}) \ge dim(V)$.
\end{observation}

\begin{definition} The \emph{excess} of a multi-flag $\cal F$ is $e({\cal F}) := \left[\displaystyle\sum_{[W_{\cal F}]\in{{\cal F}_*}} dim(W_{\cal F})\right] - dim(V)$.
\end{definition}

\begin{definition}\label{def:genpos} A multi-flag $\cal F$ in $V$ is in \emph{general position} iff  $e({\cal F}) = 0$.
\end{definition}

Any semi-flag $\cal F$ of $V$ is in general position; this is a direct consequence of the total ordering. Also the multi-flag $\cal G$ formed by a pair of subspaces $W_1, W_2\subset V$ and their common intersection (together with $\{{\bf 0}\}$ and $V$) is always in general position. For the proof of the next lemma, as well as in section 2.4 below, we will need the following definition.

\begin{definition} Let $V = (V, <\,\, ,\, >)$ be an inner product (IP) space. If $W_1\subseteq W_2\subset V$, we write\newline $(W_1\subset W_2)^\perp$ for the relative orthogonal complement of $W_1$ viewed as a subspace of $W_2$ equipped with the induced inner product (so that $W_2\cong W_1\oplus (W_1\subset W_2)^\perp$). 
\end{definition}
Note that $(W_1\subset W_2)^\perp = W_1^\perp\cap W_2$ when $W_1\subseteq W_2$ and $W_2$ is equipped with the induced inner product.

\begin{lemma}\label{lemma:2} If ${\cal G}_i$, $i = 1,2$ are two semi-flags in the inner product space $V$ and $\cal F$ is the smallest multi-flag containing ${\cal G}_1$ and ${\cal G}_2$ (in other words, it is the multi-flag generated by these two semi-flags), then $\cal F$ is in general position.
\end{lemma}
\vskip.1in
 Let ${\cal G}_i = \{W_{i,j}\}_{0\le j\le m_i}, i = 1,2$. Set $W^{j,k} := W_{1,j}\cap W_{2,k}$. Note that for each $i$, $\{W^{i,k}\}_{0\le k\le m_2}$ is a semi-flag in $W_{1,i}$, with the inclusion maps $W_{1,i}\hookrightarrow W_{1,i+1}$ inducing an inclusion of semi-flags $\{W^{i,k}\}_{0\le k\le m_2}\hookrightarrow \{W^{i+1,k}\}_{0\le k\le m_2}$. By induction on length in the first coordinate we may assume that the multi-flag of $W := W_{1,m_1-1}$ generated by $\wt{\cal G}_1 := \{W_{1,j}\}_{0\le j\le m_1-1}$ and $\wt{\cal G}_2 := \{W\cap W_{2,k}\}_{0\le k\le m_2}$ are in general position. To extend general position to the multi-flag on all of $V$, the induction step allows reduction to considering the case where the first semi-flag has only one middle term: 

\begin{claim} Given $W\subseteq V$, viewed as a semi-flag ${\cal G}'$ of $V$ of length 3, and the semi-flag ${\cal G}_2 = \{W_{2,j}\}_{0\le j\le m_2}$ as above, the multi-flag of $V$ generated by ${\cal G}'$ and ${\cal G}_2$ is in general position.
\end{claim}
\begin{proof} Without loss of generality, we may assume $V$ is equipped with an inner product. The multi-flag $\cal F$ in question is constructed by intersecting $W$ with the elements of ${\cal G}_2$,  producing the semi-flag ${\cal G}_2^W := W\cap {\cal G}_2 = \{W\cap W_{2,j}\}_{0\le j\le m_2}$ of $W$, which in turn includes into the semi-flag ${\cal G}_2$ of $V$. Constructed this way the direct-sum splittings of $W$ induced by the semi-flag $W\cap {\cal G}_2$ and of $V$ induced by the semi-flag ${\cal G}_2$ are compatible, in that if we write $W_{2,j}$ as $(W\cap W_{2,j})\oplus (W\cap W_{2,j}\subset W_{2,j})^\perp$ for each $j$, then the orthogonal complement of $W_{2,k}$ in $W_{2,k+1}$ is given as the direct sum of the orthogonal complement of $(W\cap W_{2,k})$ in $(W\cap W_{2,k+1})$ and the orthogonal complement of $(W\cap W_{2,k}\subset W_{2,k})^\perp$ in $(W\cap W_{2,k+1}\subset W_{2,k+1})^\perp$, which yields a direct-sum decomposition of $V$ indexed by the associated graded elements of $\cal F$, completing the proof both of the claim and of the lemma.
\end{proof}

On the other hand, one can construct simple examples of multi-flags which are not - in fact cannot be - in general position, as the following illustrates.

\begin{example} Let $\mathbb R\cong W_i\subset\mathbb R^2$ be three distinct 1-dimensional subspaces of $\mathbb R^2$ intersecting in the origin, and the $\cal F$ be the multi-flag generated by this data. Then $\cal F$ is not in general position.
\end{example}
\vskip.2in
\underbar{\bf Note}: Example 1 also illustrates the important distinction between a configuration of subspaces being of {\it finite type} (having finitely many isomorphism classes of configurations), and the stronger property of being in general position.

\vskip.2in
\begin{lemma}\label{lemma:generalposition} General position is preserved when closing a multiflag under
\begin{itemize}
\item subspace sums and intersections;
\item relative complements (for a fixed inner product)
\end{itemize}
Moreover, any submultiflag of a multiflag in general position is itself in general position.
\end{lemma}

\begin{proof}  These results follow from an equivalent characterization of general position in terms of special bases.
We say that $B$ is an ${\cal F}$-basis for $V$ if every non-zero subspace in ${\cal F}$ has a basis which is a subset of $B$. 
Similarly we say that $B$ is an ${\cal F}_*$-basis if $B$ can be written as the (disjoint) union over $W_{\cal F} \in {\cal F}_*$ of bases for $W_{\cal F}$. It is straightforward to see these two properties are equivalent - $B$ is an ${\cal F}$-basis iff $B$ is an ${\cal F}_*$-basis.  It is also clear that an ${\cal F}_*$ basis for $V$ exists iff ${\cal F}$ is in general position. 
\vskip.1in
Thus if $\cal F$ is in general position, we may fix a ${\cal F}$-basis $B$ for $V$. Then the closure $\wt{\cal F}$ of ${\cal F}$ under subspace sums and intersections is again in general position because $B$ is seen to also be a $\wt{\cal F}$-basis for $V$.
For any inner product on $V$ such that $V$ has an orthogonal ${\cal F}$-basis, the closure of ${\cal F}$ under relative orthogonal complements, sums, and intersections is in general position by the same argument. Finally, if $B$ is an ${\cal F}$ basis, then for any submultiflag ${\cal F}' \subset F$, $B$ is also and ${\cal F}'$ basis.
\end{proof}

\begin{corollary}\label{cor:genposiso} A multi-flag $\cal F$ on $V$ is in general position iff there is an isomorphism of vector spaces over $k$;
\begin{equation}\label{eqn:genposiso}
V\cong \bigoplus_{[W_{\cal F}]\in{{\cal F}_*}} W_{\cal F}
\end{equation}
\end{corollary}

\begin{proof} By the proof of the previous lemma, $\cal F$ is in general position iff there exists an $\cal F$-compatible basis for $V$. This latter statement is equivalent to the existence of an isomorphism as in (\ref{eqn:genposiso}).
\end{proof}
\vskip.2in

A multi-flag $\cal F$ of $V$ is a poset in a natural way; if $V_1,V_2\in {\cal F}$, then $V_1\le V_2$  as elements in $\cal F$ iff $V_1\subseteq V_2$ as subspaces of $V$.  

\begin{definition} If $\cal F$ is a multi-flag of $V$, $\cal G$ a multi-flag of $W$, a \emph{morphism} of multi-flags $(L,f):{\cal F}\to {\cal G}$ consists of
\begin{itemize}
\item a linear map from $L:V\to W$ and 
\item a map of posets $f:{\cal F}\to {\cal G}$ such that
\item for each $U\in {\cal F}$, $L(U)\subseteq f(U)$.
\end{itemize}
\end{definition}

\begin{definition} A morphism $(L,f)$ of multi-flags is \emph{closed} if for each $U\in {\cal F}, L(U) = f(U)$ (in this case the inclusion of $f$ is superfluous, and we will often write the morphism simply as $L$). $L$ is \emph{inverse-closed} if $L^{-1}(U')\in {\cal F}$ for every $U'\in {\cal G}$. It is \emph{bi-closed} if it is both closed and inverse-closed.
\end{definition}
We will denote the catgory of multi-flags and morphisms between such by $\{multi\mhyphen flags\}$.
 \vskip.2in 

Given an arbitrary collection of subspaces $T = \{W_\alpha\}$ of $V$, the multi-flag generated by $T$ is the smallest multi-flag containing each element of $T$. It can be constructed as the closure of $T$ under the operations i) inclusion of $\{{\bf 0}\}, V$ and ii) taking intersections.
\vskip.2in

If $L:V\to W$ is a linear map of vector spaces and $\cal F$ is a multi-flag of $V$, the multi-flag generated by $\{L(U)\ |\ U\in {\cal F}\}\cup \{W\}$ is a multi-flag of $W$ which we denote by $L({\cal F})$ (or $\cal F$ pushed forward by $L$). In the other direction, if $\cal G$ is a multi-flag of $W$, we write $L^{-1}[{\cal G}]$ for the multi-flag $\{L^{-1}[U]\ |\ U\in {\cal G}\}\cup \{\{{\bf 0}\}\}$ of $V$ (i.e., $\cal G$ pulled back by $L$; as intersections are preserved under taking inverse images, this will be a multi-flag once we include - if needed - $\{{\bf 0}\}$). Obviously $L$ defines morphisms of multi-flags ${\cal F}\xrightarrow{(L,\iota)} L({\cal F})$, $L^{-1}[{\cal G}]\xrightarrow{(L,\iota')} {\cal G}$. 

\begin{definition} The \emph{bi-closure} of $L$ (with respect to the multi-flags $\cal F$ and $\cal G$) is formed inductively as follows:
\begin{itemize}
\item Set ${\cal F}_0 = {\cal F}$, ${\cal G}_0 = {\cal G}$;
\item For $n\ge 0$ let ${\cal F}_{n+1}$ be the multi-flag generated by ${\cal F}_n$ and $L^{-1}({\cal G}_n)$;
\item For $n\ge 1$ let ${\cal G}_n$ be the multi-flag generated by ${\cal G}_{n-1}$ and $L({\cal F}_n)$;
\item Let ${\cal F}_\infty = \varinjlim {\cal F}_n, {\cal G}_\infty = \varinjlim {\cal G}_n$.
\end{itemize}
\end{definition}
\vskip.1in
The original morphism $L$ extends uniquely to a biclosed morphism of multi-flags $L :{\cal F}_\infty\to {\cal G}_\infty$. Moreover, it is easily seen that ${\cal F}_\infty$ and ${\cal G}_\infty$ are the smallest multi-flags containing $\cal F$ and $\cal G$ respectively for which the linear transformation $L$ induces a bi-closed morphism $L:{\cal F}_\infty\to {\cal G}_\infty$.
\vskip.2in

\begin{proposition}\label{prop:no-m-excess}
 Let $L:{\cal F}\to {\cal G}$ be a bi-closed morphism of multiflags, with $L_*:{\cal F}_*\to {\cal G}_*$ the induced map of associated gradeds. Then  $L^{-1}_*:im(L_*) \to {\cal F}_*$ is a well-defined function on $im(L_*)\backslash \{{\bf 0}\}$.
 \end{proposition}
 
 \begin{proof} Assume given $X\in {\cal F}$ with $\{{\bf 0}\}\ne [L(X)_{\cal G}]\in {\cal G}_*$. Note that $L(X)\in {\cal G}$ as $L$ is closed, and so $X' := L^{-1}(L(X))\in {\cal F}$ as $L$ is inverse-closed. There is an obvious inclusion of subspaces $X\overset{\iota}{\hookrightarrow} X'$. If this were a proper inclusion, then as both $X,X'\in {\cal F}$, $X\in S_{\cal F}(X')$, implying that the map of quotients
 \[
X/SS_{\cal F}(X)\xrightarrow{\iota} X'/SS_{\cal F}(X')
 \]
 is the zero map. However, $L(X) = L(X')$, implying
 \[
 L_*([X_{\cal F}])\xrightarrow{L_*(\iota)} L_*([X'_{\cal F}]) = [L(X)_{\cal G}]
 \]
 is the identity map. If $0 = L_*(\iota) = Id:[L(X)_{\cal G}]\to [L(X)_{\cal G}]$, then $[L(X)_{\cal G}]=\{{\bf 0}\}$, contradicting the original assumption. Thus $X=X'$. Suppose $[Y_{\cal F}]\in {\cal F}_*$ with $L_*([Y_{\cal F}]) = [L(Y)_{\cal G}]= [L(X)_{\cal G}]\ne \{{\bf 0}\}$. By Proposition \ref{prop:multi}, both $[L(X)_{\cal G}]$ and $[L(Y)_{\cal G}]$ are without multiplicity, implying $L(X) = L(Y)$. But then $L^{-1}(L(X))=L^{-1}(L(Y))$; by the above argument this implies $X=Y$.
 \end{proof}
 \vskip.2in

\begin{theorem}\label{thm:tech}
 Let $\cal F$ be a multi-flag in $V$, $\cal G$ a multi-flag in $W$, and $L:{\cal F}\to {\cal G}$ be a bi-closed morphism of multiflags induced by a linear transformation $L:V\to W$. If $[X_{\cal F}]\in {\cal F}_*$ and $L_*([X_{\cal F}]) = [L(X)_{\cal G}]\ne [{\bf 0}]\in {\cal G}_*$, then $L$ maps $X_{\cal F} = X/SS_{\cal F}(X)$ isomorphically to $L(X)/L(SS_{\cal F}(X)) = L(X)/SS_{\cal G}(L(X))$. 
 \end{theorem}
 
 \begin{proof} As $[L(X)_{\cal G}]\ne \{{\bf 0}\}$ it must also hold that $[X_{\cal F}]\ne \{{\bf 0}\}$, and so both must be without multiplicy by Proposition \ref{prop:multi}. 
 
 Clearly $L$ maps $X\subseteq V$ surjectively to its image $L(X)\subseteq W$. Because $L$ is closed, $L(SS_{\cal F}(X))\subseteq SS_{\cal G}(L(X))$. As $L$ is also inverse-closed, $X = L^{-1}(L(X))$ by the proof of Proposition \ref{prop:no-m-excess}. Hence $Z\in S_{\cal G}(L(X))$ implies $Z' = L^{-1}(Z)\in S_{\cal F}(X)$, and as $Z\subset im(L)$, $L(Z') = Z$, yielding the equality $L(SS_{\cal F}(X)) = SS_{\cal G}(L(X))$. The fact that $X=L^{-1}(L(X))$ implies $K = ker(L)\in{\cal F}$ is a proper subset of $X$, and thus $K\subseteq SS_{\cal F}(X)$. Taken together, these facts imply that the surjection $X_{\cal F} = X/SS_{\cal F}(X)\xrightarrow{L} L(X)/L(SS_{\cal F}(X)) = L(X)/SS_{\cal G}(L(X))$ must be an isomorphism.
 \end{proof}

\vskip.3in

%%%%%%%%%%%%%%%%%%%%%%%%%%%%%%%%%%%%%%%%%%%%%%%%%

\subsection{The local structure of a $\cal C$-module}
Let  $M$ be a $\cal C$-module.
\begin{definition} A \emph{multi-flag of $M$} or \emph{$M$-multi-flag}\,  is a functor $F:{\cal C}\to \{multi\mhyphen flags\}$ which assigns 
\begin{itemize}
\item to each $x\in obj({\cal C})$ a multi-flag $F(x)$ of $M(x)$;
\item to each $\phi\in Hom_{\cal C}(x,y)$ a morphism of multi-flags $F(x)\to F(y)$
\end{itemize}
\end{definition}
\vskip.1in

To any $\cal C$-module $M$ we may associate the multi-flag ${\cal F}_0(M)$ which assigns to each $x\in obj({\cal C})$ the multi-flag $\{\{{\bf 0}\}, M(x)\}$ of $M(x)$. This is referred to as the {\it trivial} multi-flag of $M$. A multi-flag on $M$ is {\it closed} resp.~{\it inverse-closed} resp.~{\it bi-closed} if that property holds for each morphism in the module.
\vskip.2in

A $\cal C$-module $M$ determines a bi-closed multi-flag on $M$. Precisely, the {\it local structure} ${\cal F}(M)$ of $M$ is defined recursively at each $x\in obj({\cal C})$ as follows: let $S_1(x)$ denote the set of morphisms of $\cal C$ originating at $x$, and $S_2(x)$ the set of morphisms terminating at $x$, $x\in obj({\cal C})$ (note that both sets contain $Id_x:x\to x$). Then
\vskip.05in
\begin{enumerate}
\item[\underbar{LS1}] ${\cal F}_0(M)(x) =$ the multi-flag of $M(x)$ generated by
\[
\{\ker(M(\phi):M(x)\to M(y))\}_{\phi\in S_1(x)}\cup \{im(M(\psi)) : M(z)\to M(x)\}_{\psi\in S_2(x)};
\]
\item[\underbar{LS2}]  For $n\ge 0$, ${\cal F}_{n+1}(M)(x) =$ the multi-flag of $M(x)$ generated by
\begin{itemize}
\item[{LS2.1}] $M(\phi)^{-1}(W)\subset M(x)$, where $W\in{\cal F}_n(M)(y)$ and $\phi\in Hom_{\cal C}(x,y)\subseteq S_1(x)$;
\item [{LS2.2}] $M(\psi)(W)\subset M(x)$, where $W\in{\cal F}_n(M)(z)$ and $\psi\in Hom_{\cal C}(z,x)\subseteq S_2(x)$;
\end{itemize}
\item [\underbar{LS3}]${\cal F}(M)(x) = \varinjlim {\cal F}_n(M)(x)$.
\end{enumerate}

More generally, starting with a multi-flag $F$ on $M$, the local structure of $M$ relative to $F$ is arrived at in exactly the same fashion, but starting in LS1 with the multi-flag $F(x)$ at each object $x$. The resulting direct limit is denoted ${\cal F}^F(M)$ (note: by definition, any multi-flag $F$ on $M$ must contain the trivial multi-flag ${\cal F}_0(M)$). Thus the local structure of $M$ (without superscript) is the local structure of $M$ relative to the trivial multi-flag on $M$, which we will denote by ${\cal F}_{-1}(M)(x)$ at each object. In almost all cases we will only be concerned with the local structure relative to the trivial multi-flag on $M$.

\begin{proposition}\label{prop:invimage} For any multi-flag $F$ on $M$, ${\cal F}^F(M)$ is the smallest bi-closed multi-flag on $M$ containing $F$.
\end{proposition}

\begin{proof} This is an immediate consequence of property (LS2). 
\end{proof}

\begin{definition} The local structure of a $\cal C$-module $M$ is the functor ${\cal F}(M):{\cal C}\to \{multi\mhyphen flags\}$, which associates to a vertex $x\in obj({\cal C})$ the multi-flag ${\cal F}(M)(x)$, and a morphism $\phi\in Hom_{\cal C}(x,y)$ the induced bi-closed morphism of multi-flags ${\cal F}(\phi): {\cal F}(M)(x)\to {\cal F}(M)(y)$.
\end{definition}

A key question arises as to whether the direct limit used in defining ${\cal F}(M)(x)$ stablizes at a finite stage. This motivates

\begin{definition} The local structure on $M$ is \emph{locally stable} at $x\in obj({\cal C})$ iff there exists $N = N_x$ such that ${\cal F}_n(M)(x)\inj {\cal F}_{n+1}(M)(x)$ is the identity map whenever $n\ge N$. It is \emph{stable} if it is locally stable at each object. It is \emph{strongly stable} if for all \emph{finite} multi-flags $F$ on $M$ there exists $N = N(F)$ such that ${\cal F}^F(M)(x) = {\cal F}^F_N(M)(x)$ for all $x\in obj({\cal C})$.
\end{definition}

In almost all applications of this definition we will only be concerned with stability, not the related notion of strong stability.
\vskip.2in

The following result identifies the effect of a morphism in $M$ on the associated graded limit ${\cal F}(M)_*$. 

\begin{theorem}\label{thm:1} Let $M$ be a $\cal C$-module with stable local structure. Then for all $x,y,z\in obj({\cal C})$, $W\in {\cal F}(M)(x)$, $\phi\in Hom_{\cal C}(z,x)$, and $\psi\in Hom_{\cal C}(x,y)$
\begin{enumerate}
\item The morphisms of $M$  induce well-defined maps of associated graded sets
\[
M(\psi)_*:{\cal F}(M)_*(x)\to {\cal F}(M)_*(y)
\]
while their inverses yield multi-functions of associated graded sets
\[
M(\phi)^{-1}_*: im(M(\phi)_*)\to {\cal F}(M)_*(z)
\]
with $M(\phi)^{-1}_*$ defining a function on $im(M(\phi)_*)\backslash\{{\bf 0}\}$.
\item $M(\psi)(W)\in {\cal F}(M)(y)$, and either $M(\psi)(W_{{\cal F}(M)(x)}) = \{{\bf 0}\}$, or $M(\psi):W_{{\cal F}(M)(x)}\xrightarrow{\cong}M(\psi)(W_{{\cal F}(M)(x)}) = (M(\psi)(W))_{{\cal F}(M)(y)}$;
\item if $[W_{\cal F}]\ne \{{\bf 0}\}$, then either $[W_{{\cal F}(M)(x)}]\notin im(M(\phi)_*)$, or there is a unique element $U \in M(\phi)^{-1}_*( im(M(\phi)_*))\subset {\cal F}(M)_*(z)$ with $M(\phi):U_{{\cal F}(M)(z)}\xrightarrow{\cong} W_{{\cal F}(M)(x)}$.
\end{enumerate}
\end{theorem}

\begin{proof} Given that ${\cal F}(M)$ is bi-closed with respect to each morphism in $\cal C$, these three statements are an immediate consequence of Theorem \ref{thm:tech}.
\end{proof}
\vskip.1in

We will use the notion of general position, discussed above, to define excess.
\vskip.1in

{\bf\underbar{Excess for local structure}} The {\it excess} of a $\cal C$-module $M$ is
\[
e(M) = \sum_{x\in obj({\cal C})} e({\cal F}(M)(x))
\]

\begin{definition} The local structure ${\cal F}(M)$ is in \emph{general position} iff $e(M) = 0$.
\end{definition}
\vskip.1in

Note that as $M(x)$ is finite-dimensional for each $x\in obj({\cal C})$, ${\cal F}(M)(x)$ must be locally stable at $x$ if it is in object general position (although general position is a much more restrictive property).
\vskip.3in

%%%%%%%%%%%%%%%%%%%%%%%%%%%%%%%%%%%%%%%%%%%%%%%%%

\subsection{Inner product structures} It will be useful to consider two refinements of the category $(vect/k)$.
\begin{enumerate}
\item $(WIP/k)$, the category whose objects are inner product (IP)-spaces $V = (V,<\ ,\ >_V)$ and whose morphisms are linear transformations (no compatibility required with respect to the inner product structures on the domain and range);
\item $(PIP/k)$, the wide partial subcategory of $(WIP/k)$ whose morphisms
\[
L:(V,<\ ,\ >_V)\to (W,<\ ,\ >_W)
\]
are {\it partial isometries}; that is, $\wt{L}: ker(L)^\perp\to W$ is an isometric embedding, where $ker(L)^\perp\subset V$ denotes the orthogonal complement of $ker(L)\subset V$ in $V$ with respect to the inner product $<\ ,\ >_V$, and $\wt{L}$ is the restriction of $L$ to $ker(L)^\perp$.
\end{enumerate}

[{\bf\underbar{Note of explanation}}: A {\it partial subcategory} (aka pre-subcategory) ${\cal C}$ of a category ${\cal D}$ refers to i) a collection of objects $obj({\cal C})\subset obj({\cal D})$ and morphisms $hom({\cal C})\subset hom({\cal D})$ for which the composition of morphisms in $\cal C$, which exists in $\cal D$, may not lie in $\cal C$. A fundamental complication in dealing with partial isometries vs isometries is that they need not be closed under composition - a good discussion of this issue is given in \cite{fp}. Partial categories arise in other contexts; for example partial monoids as defined in \cite{gs} may be viewed as a partial category with a single object.]
\vskip.2in
There are obvious transformations
\[
(IP/k)\xrightarrow{\iota_{ip}} (WIP/k)\xrightarrow{p_{wip}} (vect/k)
\]
where the first map is the inclusion which is the identity on objects, while the second map forgets the inner product on objects and is the identity on transformations between two fixed objects. We will call $\cal D$ an {\it IP-category} if ${\cal D}$ is a subcatgory of the partial category $(PIP/k)$. 
\vskip.2in

\begin{definition}
Given a $\cal C$-module $M:{\cal C}\to (vect/k)$ a \emph{weak inner product} on $M$ is a factorization 
\[
M: {\cal C}\to (WIP/k)\xrightarrow{p_{wip}} (vect/k)
\]
while an \emph{inner product} on $M$ is a further factorization through (or lift to) an $IP$ category ${\cal D}\subset (IP/k)$:
\[
M: {\cal C}\to {\cal D}\hookrightarrow (IP/k)\xrightarrow{\iota_{ip}}(WIP/k)\xrightarrow{p_{wip}} (vect/k)
\]
\end{definition}

A $\cal WIPC$-module will refer to a $\cal C$-module $M$ equipped with a weak inner product, while an $\cal IPC$-module is a $\cal C$-module that is equipped with an actual inner product, in the above sense. As any vector space admits a (non-unique) inner product, we see that

\begin{proposition} Any $\cal C$-module $M$ admits a non-canonical representation as a $\cal WIPC$-module.
\end{proposition}

The question as to whether a $\cal C$-module $M$ can be represented as an $\cal IPC$-module, however, is a much more delicate issue.
\vskip.2in

Given a $\cal C$-module $M$ and a morphism $\phi\in Hom_{\cal C}(x,y)$, we set $KM_{\phi} := \ker(\phi : M(x)\to M(y)).$ We note that a $\cal C$-module $M$ is an $\cal IPC$-module, iff
\begin{itemize}
\item for all $x\in obj({\cal C})$, $M(x)$ comes equipped with an inner product $<\,\, ,\, >_x$;
\item for all $\phi\in Hom_{\cal C}(x,y)$, the map $\wt{\phi} : KM_{\phi}^\perp\to M(y)$ is an isometric embedding, where $\wt{\phi}$ denotes the restriction of $\phi$ to $KM_{\phi}^\perp = $ the orthogonal complement of $KM_{\phi}\subset M(x)$ with respect to the inner product $<\,\, ,\, >_x$. In other words, 
\[
<\phi({\bf v}), \phi({\bf w})>_y = <{\bf v}, {\bf w}>_x,\qquad \forall\, {\bf v}, {\bf w}\in KM_{\phi}^\perp
\]
\end{itemize}

\begin{definition} Given a multi-flag $\cal F$ on an inner product space $(V, <\,\, ,\, >)$, we will say the inner product is \emph{$\cal F$-compatible} - and that $V$ is an $\cal F$IP-space - if every q-isomorphism $\iota:(W, SS_{\cal F}(W))\to (W',SS_{\cal F}(W'))$ induces, upon restriction to relative orthogonal complements, an isomorphism
\begin{equation}\label{eqn:fipc}
\wt{W}_{\cal F} := (SS_{\cal F}(W)\subset W)^\perp\xrightarrow[\cong]{\iota} (SS_{\cal F}(W'),W')^\perp = \wt{W}'_{\cal F}
\end{equation}
An $\cal FIPC$-module $M$ is an $\cal IPC$-module for which the inner product $<\,\, ,\, >_x$ on $M(x)$ is ${\cal F}(M)(x)$-compatible for each $x\in obj({\cal C})$.
\end{definition}
\vskip.2in
For an arbitrary multi-flag on $M$ that is not necessarily biclosed, admitting an $\cal FIPC$-structure is in general a strictly stronger condition than admitting an $\cal IPC$-structure. However these notions turn out to be equivalent for the local structure of a $\cal C$-module $M$.
\vskip.2in

\begin{lemma} If $M$ is an $\cal IPC$-module with stable local structure, then it is an $\cal FIPC$-module.
\end{lemma}

\begin{proof} Let $W\in {\cal F}(M)(x)$. Then $[W_{\cal F}] = \{{\bf 0}\}$ iff $(SS_{\cal F}(W)\subset W)^\perp = \{{\bf 0}\}$, for which the condition in (\ref{eqn:fipc}) is trivially satisfied. On the other hand, if $[W_{\cal F}]\ne \{{\bf 0}\}$, Proposition \ref{prop:multi} implies the q-isomorphism $\iota$ in the above definition must be the identity map, again implying $\cal F$-compatibility for trivial reasons.
\end{proof}
For this reason, in what follows we will only need to assume an $\cal IPC$-structure on $M$, as the $\cal F$-compatibility of that structure comes along for free. Assume now that $M$ is equipped with such a structure. In this case all morphisms $M(\phi):M(x)\to M(y)$ are partial isometries which preserve relative orthogonal complements: 
\[
\wt{W}_{\cal F}=(SS_{{\cal F}(M)(x)}(W)\subset W)^\perp\xrightarrow{M(\phi)}(SS_{{\cal F}(M)(y)}(M(\phi)(W))\subset M(\phi)(W))^\perp
\]
This map is either $0$ or an isometry by the same argument appearing in the proof of Theorem \ref{thm:1}; moreover, these relative orthogonal complements remain invariant under q-isomorphisms. A similar analysis applies for inverse images. Consequently, we can lift the arguments given in Theorem \ref{thm:1} from subquotients to orthogonal complements, yielding

\begin{theorem}\label{thm:2} Let $M$ be an $\cal IPC$-module with stable local structure. Then for all $x,y,z\in obj({\cal C})$, $W\in {\cal F}(M)(x)$, $\phi\in Hom_{\cal C}(z,x)$, and $\psi\in Hom_{\cal C}(x,y)$
\begin{enumerate}
\item The morphism $M(\psi)$ induces a well-defined map of associated graded sets
\[
M(\psi)_*:{\cal F}(M)_*(x)\to {\cal F}(M)_*(y)
\]
while the inverse of $M(\phi)_*$ yields a multi-function
\[
M(\phi)^{-1}_*: im(M(\phi)_*)\to {\cal F}(M)_*(z)
\]
which is a function on $im(M(\phi)_*)\backslash\{{\bf 0}\}$.
\item $M(\psi)(W)\in {\cal F}(M)(y)$, and either $M(\psi)(\wt{W}_{{\cal F}(M)(x)}) = \{{\bf 0}\}$, or $M(\psi):\wt{W}_{{\cal F}(M)(x)}\xrightarrow{\cong}M(\psi)(\wt{W}_{{\cal F}(M)(x)}) = \wt{(M(\psi)(W))}_{{\cal F}(M)(y)}$;
\item if $[W_{{\cal F}(M)(x)}]\ne\{{\bf 0}\}$, then either $[W_{{\cal F}(M)(x)}]\notin im(M(\phi)_*)$, or there is a unique element $U = M(\phi)^{-1}(W)\in {\cal F}(M)(z)$ with $M(\phi)_*([U_{{\cal F}(M)(z)}]) = [W_{{\cal F}(M)(x)}]$ and  $M(\phi):\wt{U}_{{\cal F}(M)(z)}\xrightarrow{\cong} \wt{W}_{{\cal F}(M)(x)}$.
\end{enumerate}
\end{theorem}

\begin{proof} Same as above.
\end{proof}
\vskip.3in

%%%%%%%%%%%%%%%%%%%%%%%%%%%%%%%%%%%%%%%%%%%%%%%%%

\subsection{Sums, Tensors, and the Algebraic K\"unneth Theorem}

Let $\Delta \cal{C}$ be the full subcategory  of $\cal{C} \times \cal{C}$ whose object set is the diagonal of $\cal{C} \times \cal{C}$.
We define the direct sum of 2 basepointed modules $M:{\cal C}_*\to (vect/k)_*$ and $N:{\cal C}_*\to (vect/k)_*$ to be the functor $M\oplus N:\Delta{\cal C}_*\to (vect/k)_*$ defined componentwise on objects and morphisms. Due to the necessity of introducing zero maps, this definition is reminiscent of the definition of the Whitney Sum for vector bundles.
We define the direct sum of multiflags to be the multiflag of pairwise direct sums, i.e. ${\cal F} \oplus {\cal G}:=\{U \oplus W \ \mid \ U \in {\cal F}, \ W \in {\cal G}\}$

\vskip.2in 
We will denote by $\{{\bf 0}\}$ the trivial multiflag on the zero vector space.
\vskip.2in

\begin{lemma}\label{lemma:sumofmultiflags} The associated graded of the sum of two multiflags is the union of each of their respective associated gradeds. i.e. 
$
({\cal F} \oplus {\cal G})_* = ({\cal F})_* \oplus \{{\bf 0}\} \ \cup \ \{{\bf 0}\} \oplus ({\cal G})_*
$
\end{lemma}

\begin{proof}
Every subspace in ${\cal F} \oplus {\cal G}$ of the form $U \oplus W$, $U \neq \{{\bf0}\} \neq W$ contains the subspaces $U \oplus \{{\bf 0}\}$ and $\{{\bf 0}\} \oplus W$ as proper subspaces, hence is zero in the associated graded.
\end{proof}
\vskip.2in

While the local structure for basepointed modules does not distribute over direct sums on the level of multiflags, it does at the level of the associated graded. In particular, the associated graded of a stable local structure distributes over direct sums.

\begin{lemma} Let $M, N$ be basepointed as above. For $i\geq -1$, the following hold
\begin{enumerate}
\item ${\cal F}_i(M \oplus N)(x) 
	\ \subseteq \ 
	{\cal F}_i(M)(x)
	\ \oplus \ 
	{\cal F}_i(N)(x)$
\item ${\cal F}_i(M \oplus N)_*(x) 
	\ = \ 
	\left(
	{\cal F}_i(M)(x)
	\ \oplus \ 
	{\cal F}_i(N)(x)
	\right)_*$
\item ${\cal F}(M \oplus N)_*(x) 
	\ = \ 
	\left(
	{\cal F}(M)(x)
	\ \oplus \ 
	{\cal F}(N)(x)
	\right)_* = {\cal F}(M)_*(x)
	\ \oplus \ 
	{\cal F}(N)_*(x)$ if the local structures of $M$ and $N$ are stable.
\end{enumerate}
\end{lemma}
\begin{proof}
1 follows from the construction of direct sums of modules and of local structure. 2 follows from 1 by applying Lemma \ref{lemma:sumofmultiflags}. 3 is a restatement of 2 when $i=\infty$.
\end{proof}
\vskip.2in

We now return to the normal unbasepointed setting. 

\begin{definition} The tensor product of two modules $M:{\cal C}\to (vect/k)$ and $N:{\cal D}\to (vect/k)$ is the functor $M\otimes N:{\cal C} \times {\cal D}\to (vect/k)$ given componentwise on both objects and morphisms. In the same vein, we define the tensor product of multiflags to be the multiflag of pairwise tensor products, i.e. ${\cal F} \otimes {\cal G}:=\{U \otimes W \ \mid \ U \in {\cal F}, \ W \in {\cal G}\}$.
\end{definition}
\vskip.2in

\begin{lemma}\label{lemma:tensorofmultiflags} The associated graded distributes over tensor products, i.e.
$
({\cal F} \otimes {\cal G})_* = ({\cal F})_* \times ({\cal G})_*
$
\end{lemma}

\begin{proof} We first note that
$$
\frac{U \otimes W}{SS_{{\cal F} \otimes {\cal G}}(U \otimes W)}
=
\frac{U}{SS_{\cal F} U} \otimes \frac{W}{SS_{\cal G} W}.
$$

As
$$
\frac{U}{SS_{\cal F} U} \otimes \frac{W}{SS_{\cal G} W}
=
\frac{U \otimes W}{(U \otimes SS_{\cal G} W) \ + \ (SS_{\cal F} U \otimes W)}
$$
the desired equality follows from
$$
SS_{{\cal F} \otimes {\cal G}}(U \otimes W)
=
(U \otimes SS_{\cal G} W) \ + \ (SS_{\cal F} U \otimes W).
$$
In a similar vein, these equalities show that $(U,SS_{\cal F}(U))\hookrightarrow (U',SS_{\cal F}(U'))$ and $(W,SS_{\cal G}(W))\hookrightarrow (W',SS_{\cal G}(W'))$ are q-isomorphisms iff $(U\otimes W,SS_{{\cal F}\otimes{\cal G}}(U\otimes W))\hookrightarrow (U'\otimes W',SS_{{\cal F}\otimes{\cal G}}(U'\otimes W'))$ is one as well. The result follows.
\end{proof}
\vskip.2in

\begin{theorem}\label{thm:kunneth} (K\"unneth Theorem for Local Structure) The local structure of $M\otimes N$ distributes over the tensor product at every finite stage (hence so does the stable local structure) both at the level of multiflags and at the level of the associated graded. i.e. for all $i\geq -1$ and $i = \infty$,
\begin{enumerate}
\item $
	{\cal F}_{i}(M \otimes N)(x,x')
	=
	{\cal F}_{i}(M)(x) \otimes {\cal F}_i(N)(x')
	$
\item $
	{\cal F}_{i}(M \otimes N)_*(x,x')
	=
	{\cal F}_{i}(M)_*(x) \times {\cal F}_i(N)_*(x')
	$
\end{enumerate}

\end{theorem}
\begin{proof} Proceed by induction. When $i=-1$, we have:
\[
{\cal F}_{-1}(M \otimes N)(x,x') = \{{\bf 0}, M(x) \otimes N(x')\} = {\cal F}_{-1}(M)(x) \otimes {\cal F}_{-1}(N)(x')
\]
Suppose the claim holds for $i$. Let $\phi_{ab}\in Hom_{\cal C}(a,b)$ denote an arbitrary morphism. Then ${\cal F}_{i+1}(M)(x) \otimes {\cal F}_{i+1}(N)(x')$ is generated (via closure under intersections) by the following four sets of images and inverse images:
$$\{  M(\phi_{zx})(U_z) \otimes N(\phi_{z'x'})(W_{z'}) \}$$
$$\{  M(\phi_{zx})(U_z) \otimes N(\phi_{x'y'})\inv(W_{y'}) \}$$
$$\{  M(\phi_{xy})\inv(U_y) \otimes N(\phi_{z'x'})(W_{z'}) \}$$
$$\{  M(\phi_{xy})\inv(U_y) \otimes N(\phi_{x'y'})\inv(W_{y'}) \}$$
where $U_a \in {\cal F}_{i}(M)(a)$ and $W_b \in {\cal F}_{i}(N)(b)$. By the induction hypothesis, ${\cal F}_{i+1}(M \otimes N)(x, x')$ is generated in the same manner by the following two sets:
$$ \{ M(\phi_{zx})(U_z) \otimes N(\phi_{z'x'})(W_{z'})   \} $$
$$ \{ M(\phi_{xy})\inv(U_y) \otimes N(\phi_{x'y'})\inv(W_{y'})   \} $$
These two descriptions clearly imply ${\cal F}_{i+1}(M \otimes N)(x, x') \subseteq {\cal F}_{i+1}(M)(x) \otimes {\cal F}_{i+1}(N)(x')$. Inclusion in the other direction follows from the equalities
$$ M(\phi_{zx})(U_z) \otimes N(\phi_{x'y'})\inv(W_{y'}) = \left(M(\phi_{zx})(U_z)\otimes N(x')\right) \cap \left(M(x)\otimes N(\phi_{x'y'})\inv(W_{y'})\right)  $$
$$ M(\phi_{xy})\inv(U_y) \otimes N(\phi_{z'x'})(W_{z'}) = \left(M(\phi_{xy})\inv(U_y)\otimes N(x')\right)\cap \left(M(x)\otimes N(\phi_{z'x'})(W_{z'})\right)$$
Statement 2.~follows from 1.~by the previous lemma.
\end{proof}
\vskip.5in

%%%%%%%%%%%%%%%%%%%%%%%%%%%%%%%%%%%%%%%%%%%%%%%%%
%%%%%%%%%%%%%%%%%%%%%%%%%%%%%%%%%%%%%%%%%%%%%%%%%

\section{Coverings} We construct - for any $\cal C$-module $M$ - a {\it quasi-tame} module (defined below) which covers it in an appropriate sense. We also establish a sufficient condition for the quasi-tame covering to be b-tame (also defined below).
\vskip.3in

%%%%%%%%%%%%%%%%%%%%%%%%%%%%%%%%%%%%%%%%%%%%%%%%%

\subsection{The quasi-tame covering of a $\cal C$-module}
Let $M$ be a $\cal C$-module with stable local structure. 

\begin{definition} Fixing an object $x\in obj({\cal C})$ and an element $\{{\bf 0}\}\ne [W_{{\cal F}(M)(x)}]\in {\cal F}(M)_*(x)$ with unique representative $W_{{\cal F}(M)(x)}\in {\cal F}(M)(x)$, $Supp(W_{{\cal F}(M)(x)})$ is the smallest $\cal C$-module satisfying
\begin{itemize}
\item ${\cal F}(Supp(W_{{\cal F}(M)(x)}))_*$ is a sub-$\cal C$-set of ${\cal F}(M)_*$
\item $W_{{\cal F}(M)(x)}\subset Supp(W_{{\cal F}(M)(x)})(x)$
\end{itemize}
\end{definition}
\vskip.1in

Precisely, for each object $y\in obj({\cal C})$ and element $U_{{\cal F}(M)(y)}\in {\cal F}(M)(y)$, $U_{{\cal F}(M)(y)}\in {\cal F}(Supp(W_{{\cal F}(M)(x)}))(y)$ iff $U_{{\cal F}(M)(y)}$ is connected to $W_{{\cal F}(M)(x)}$ by a zig-zag sequence of isomorphisms in ${\cal F}(M)_*$ induced by restrictions to subquotients of morphisms in the $\cal C$-module $M$. We refer to a $\cal C$-module arising in this fashion (i.e., isomorphic to $Supp(W_{{\cal F}(M)(x)})$ for some $\{{\bf 0}\}\ne [W_{{\cal F}(M)(x)}]$) as a \emph{quasi-block}.
\vskip.1in

 Suppose $x\ne y\in obj(Supp(W_{{\cal F}(M)(x)}))$ and $U_{{\cal F}(M)(y)}$ is the unique representative of the equivalence class $[U_{{\cal F}(M)(y)}]$ in ${\cal F}(M)_*(y)$ connected by a zig-zag sequence of isomorphisms to $[W_{{\cal F}(M)(x)}]\in {\cal F}(M)_*(x)$. There is then a canonical identification of $\cal C$-modules
\begin{equation}
Supp(W_{{\cal F}(M)(x)})\cong Supp(U_{{\cal F}(M)(y)})
\end{equation}
When two $\cal C$-quasi-blocks are identified in this fashion (including the case $x=y$), we will refer to them as being {\it quasi-block-equivalent}, indicated by the symbol $\underset{B}{\sim}$.

\begin{definition} A $\cal C$-module is {\it quasi-tame} if it can be written as a finite direct sum of $\cal C$-quasi-blocks. The {\it quasi-tame cover} of $M$ is the $\cal C$-module
\begin{equation}\label{eqn:quasi-cover}
{\cal QTC}(M) := \left(\bigoplus_{x\in obj({\cal C})}\left(\bigoplus_{\{{\bf 0}\}\ne [W_{{\cal F}(M)(x)}]\in {\cal F}(M)_*(x)}\ Supp\left(W_{{\cal F}(M)(x)}\right)\right)\right){\Bigg/}\mathlarger{\mathlarger{\underset{B}{\sim}}}
\end{equation}
\end{definition}
\vskip.1in

\begin{theorem}\label{thm:5} There is a natural isomorphism $P(M)_*: {\cal F}({\cal QTC}(M))_*\xrightarrow{\cong} {\cal F}(M)_*$ of associated graded local structures given on elements by
\[
{\cal F}({\cal QTC}(M))_*(x)\ni [W_{{\cal F}(M)(x)}]\overset{P(M)_*}{\mathlarger{\longmapsto}} [W_{{\cal F}(M)(x)}]\in {\cal F}(M)_*(x)
\]

In addition, ${\cal QTC}(M)$ has stable local structure, and is in general position: $e({\cal QTC}(M)) = 0$. 
\end{theorem}

\begin{proof} The first statement follows directly from the construction of ${\cal QTC}(M)$. Any quasi-block associated to a $\cal C$-module with stable local structure will have stable local structure, and has vanishing excess (again, by construction). Therefore ${\cal QTC}(M)$, which is a finite direct sum of quasi-blocks, satisfies the same properties. 
\end{proof}
\vskip.3in

%%%%%%%%%%%%%%%%%%%%%%%%%%%%%%%%%%%%%%%%%%%

\subsection{Tame coverings} A $\cal C$-module $N$ is a {\it block} (referred to as a $\cal C$-block) if whenever $N(x)\ne \{{\bf 0}\}\ne N(y)$, there is a zig-zag sequence of isomorphisms in $N$ connecting $N(x)$ and $N(y)$. We say a $\cal C$-module is {\it b-tame} if it is a finite direct sum of blocks. If every $Supp(W_{{\cal F}(M)(x)})$ in (\ref{eqn:quasi-cover}) above is a block, the quasi-tame cover is b-tame, and referred to as the {\it b-tame cover}, written ${\cal TC}(M)$.

\begin{lemma}\label{lemma:7} Suppose $M$ is a $\cal C$-module such that for all $x\in obj({\cal C})$, $\{{\bf 0}\}\ne [W_{{\cal F}(M)(x)}]\in {\cal F}_*(M)_*(x)$ and zig-zag composition of isomorphisms ${\cal F}(M)(x)\ni W_{{\cal F}(M)(x)}\overset{\cong}{\leftrightarrow} U_{{\cal F}(M)(x)}\in{\cal F}(M)(x)$ (induced by morphisms in $M$), $W_{{\cal F}(M)(x)} = U_{{\cal F}(M)(x)}$. Then each quasi-block in the quasi-tame cover is a block, and so ${\cal QTC}(M) = {\cal TC}(M)$.
\end{lemma}

\begin{proof} The condition implies that $[W_{{\cal F}(M)(x)}]$ cannot be connected by a zig-zag sequence of isomorphisms to an element $[U_{{\cal F}(M)(x)}]\ne [W_{{\cal F}(M)(x)}]$, forcing $Supp(W_{{\cal F}(M)(x)})$ to satisfy the precise conditions needed to be a block. The result follows.
\end{proof}
\vskip.2in

It is currently unknown if the isomorphism $P_*(M)$ of associated graded local structures is always induced by a morphism of $\cal C$-modules ${\cal QTC}(M)\to M$ (although it is if $M$ is equipped with an $\cal IPC$-structure, as discussed below). Nevertheless, the following is clear

\begin{lemma}\label{lemma:8} If $M$ is a quasi-tame $\cal C$-module, then ${\cal QTC}(M) = M$, with the isomorphism $P(M)_*$ induced by this equality. Moreover if $M$ is b-tame, then ${\cal QTC}(M) = {\cal TC}(M)$.
\end{lemma}
\vskip.3in

%%%%%%%%%%%%%%%%%%%%%%%%%%%%%%%%%%%%%%%%%%%%%%%%%

\subsection{Coverings of $\cal IPC$-modules} We assume now that $M$ is an $\cal IPC$-module. In this case for each $x\in obj({\cal C})$ and $(W,SS_{{\cal F}(M)(x)}(W))\in {\cal F}(M)(x)^p$, the canonical isomorphism
\begin{equation}\label{eqn:embed}
(SS_{{\cal F}(M)(x)}(W)\subset W)^\perp\xrightarrow{\cong} W_{{\cal F}(M)(x)}
\end{equation}
provides an embedding $W_{{\cal F}(M)(x)}\hookrightarrow M(x)$ compatible with respect to taking direct and inverse images of morphisms of $M$ mapping from or to $M(x)$. The result, again for each $x\in obj({\cal C})$ and $[W_{{\cal F}(M)(x)}]\in {\cal F}(M)_*(x)$, is an extension of (\ref{eqn:embed}) to an embedding of $\cal C$-modules
\[
Supp(W_{{\cal F}(M)(x)})\hookrightarrow M
\]

\begin{theorem} For an $\cal IPC$-module $M$, the isomorphism $P(M)_*$ of associated graded local structures in Theorem \ref{thm:5} is induced by a surjection of $\cal C$-modules
\begin{equation}\label{eqn:proj}
P(M): {\cal QTC}(M)\surj M
\end{equation}
This surjection is an isomorphism (i.e., $M$ is quasi-tame) iff $e(M) = 0$.  Moreover, if in addition all of the quasi-block summands of ${\cal QTC}(M)$ are blocks, then $M$ is b-tame iff $e(M) = 0$.
\end{theorem}

\begin{proof} Choosing a representative for each $B$-equivalence class and summing these inclusions as in (\ref{eqn:quasi-cover}) produces a the required surjection $P(M)$ of $\cal C$-modules by the above. The total dimension of the kernel $ker(P(M))$ is exactly the object excess $e(M)$, implying the second statement. The final statement follows from the results of the previous section.
\end{proof}
\vskip.3in

%%%%%%%%%%%%%%%%%%%%%%%%%%%%%%%%%%%%%%%%%%%%%%%%%

\subsection{Decomposition of 1-dimensional persistence modules via local structure} We adopt the following (slightly) restrictive definition.

\begin{definition} A \emph{finite 1-dimensional persistence module over $k$} is a functor $M:\underline{\rm n}\cong{\cal C}\to (vect/k)$ where $\underline{\rm n}$ denotes the categorical realization of the totally ordered set $1 < 2 < 3<\dots < n$. 
\end{definition}
We will need some additional terminology.
\begin{definition} Given $1\le i<j\le n$, the \emph{interval [i,j]} is the full subcategory of $\underline{\rm n}$ on objects $\{k\in \mathbb N\ |\ i\le k\le n\}$ (using this notation, $\underline{\rm n}$ can alternatively be written as $[1,n]$). If $M$ is an $\underline{\rm n}$-module, and $1\le i < j\le n$, $M$ is an \emph{[i,k]-block} if it is a block, and $M(j)\ne \{0\}$ iff $i\le j\le k$. M is an \emph{interval module of type [i,k]} if it is an $[i,k]$-block and $dim_k(M(j))\le 1$ for all $1\le j\le n$.
\end{definition}

\vskip.1in
It is a consequence of Gabriel's Theorem \cite{pg} that persistence modules decompose as a direct sum of interval submodules and that the decomposition is unique up to reordering.
\vskip.2in
In this section we give an alternative proof of this result using local structure. 

\begin{lemma} Any persistence module $M:\underline{n}\to (vect/k)$ admits an inner product structure, hence can be realized as an ${\cal I}\underline{n}$-module.
\end{lemma}
\begin{proof} The first statement is true trivially for $n=1$. For $1\le i\le j\le n$ denote the unique morphism from $M(i)$ to $M(j)$ by $\phi_{i,j}$. By induction, we may assume that we have fixed an inner product structure on the sub-module $M' := \{M(2)\xrightarrow{\phi_{1,2}} M(3)\to\dots\xrightarrow{\phi_{(n-1),n}} M(n)\}$.  For $1 < j\le n$ let $\ker(i,j) := \ker(\phi_{i,j})$. This yields the semiflag
\[
\ker(1,2)\subseteq \ker(1,3)\subseteq\dots\subseteq \ker(1,n)\subseteq M(1)
\]
Next, choose a direct sum decomposition
\[
M(1)\cong \ker(1,2)\oplus \ker(1,3)/\ker(1,2)\oplus \ker(1,4)/\ker(1,3)\oplus\dots\oplus \ker(1,n)/\ker(1,n-1)\oplus M(1)/\ker(1,n)
\]
We observe that $\ker(1,j)/\ker(1,j-1)$ maps isomorphically to its image in $M(j-1)$ under $\phi_{1,(j-1)}$, and then to $\{{\bf 0}\}\subset M(j)$ under $\phi_{(j-1),j}$. On this summand choose the inner product to be that induced by the one on $M(j-1)$ via the embedding $\phi_{1,(j-1)}$; similarly equip $M(1)/\ker(1,n)$ with the one induced by its embedding via $\phi_{1,n}$ into $M(n)$. Choose an arbitrary inner product on $\ker(1,2)$. This collection of inner products extends to a unique inner product on $M(1)$ which agrees with the previously defined inner product on each summand, and where the summands themelves are mutually orthogonal. By construction, each $\phi_{1,j}$ will be a partial isometry, yielding the requisite extension of the $\cal IP$-structure on $M'$ to $M$.
\end{proof}
\vskip.2in

\begin{lemma}\label{lemma:10} If a finite 1-dim persistence module $M:\underline{n}\to (vect/k)$ is a quasi-block, it is a block. Moreover, if $M$ is a block, then $supp(M) = [k,l]$ for some $1\le k\le l\le n$.
\end{lemma}

\begin{proof} Associate the symbol $x_i$ with the morphism $\phi_{i,i+1}, 1\le i\le n-1$, and let $F_{n-1}$ be the free group on $x_1,\dots,x_{n-1}$. Via this association, we see that a zig-zag sequence of isomorphisms beginning and ending at $M(k)$ corresponds in $F_{n-1}$ to a word $w$ which in reduced form equals $e$ ($e$ denoting the identity element). This implies the self-map $M(i)\xrightarrow{\cong} M(i)$ must be the identity, and hence the quasi-block is a block. The support of this block must be a connected subcategory of $\underline{n}$, implying it is an interval subcategory of the form $[k,l]$ as claimed.
\end{proof}
\vskip.2in

For convenience we will abbrieviate $\im(\phi_{i,j})$ as $\im(i,j)$. A fact we will exploit extensively in the following discussion is that at each object the semiflag of images and kernels are each nested: 
\[
\im(0,k) \subseteq \cdots \subseteq \im(k,k)
\]
\[
\ker(k,k) \subseteq \cdots \subseteq \ker(k,n)
\]

Let 
$$
	{\cal I}(k) :=
\left\{
	\bigcap_{l=1}^m 
	\left[ 
		\im(i_l,k) + \ker(k,j_l)
	\right]
	\ | \
	m \in \mathbb N, \ 
	i_l \in [0,k], \
	j_l \in [k,n+1]
\right\}
$$
$$
	{\cal S}(k) := 
	\left\{
	\sum_{l=1}^m 
	\left[ 
		\im(i_l,k) \cap \ker(k,j_l)
	\right]
	\ | \
	m \in \mathbb N, \ 
	i_l \in [0,k], \
	j_l \in [k,n+1]
\right\}
$$	

By construction, ${\cal S}(k)$ is closed and ${\cal I}(k)$ is inverse closed. Since the semiflags of kernels and images are nested, each element of ${\cal I}(k)$ and ${\cal S}(k)$ can be written in multiple ways. We will want a way for eliminating obviously redundant subspaces. For $i\le j\in\mathbb N$,  $S[i,j]$ will denote the set of integers $i\le k\le j$ equipped with its standard total ordering. This induces the usual partial ordering on the Cartesian product $S[i,j]\times S[i',j']$. We recall that an \emph{anti-chain} in a poset is a subset in which each distinct pair of elements are unrelated (incomparable).
\vskip.2in

 We now consider the following subcollections:

\begin{align*}
	{\cal I}_r(k) :=
\Bigg\{
	\bigcap_{l=1}^m 
	\left[ 
		\im(i_l,k) + \ker(k,j_l)
	\right]
	\ \vrule \ 
	& m \in \mathbb N, \ 
	(i_l,j_l) \in [0,k] \times [k,n+1], \\
	& \{(i_l,j_l)\}_{l=1}^m \textnormal{ is an antichain.}
\Bigg\}
\end{align*}
\begin{align*}
	{\cal S}_r(k) :=
\Bigg\{
	\sum_{l=1}^m 
	\left[ 
		\im(i_l,k) \cap \ker(k,j_l)
	\right]
	\ \vrule \ 
	& m \in \mathbb N, \ 
	(i_l,j_l) \in [0,k] \times [k,n+1], \\
	& \{(i_l,j_l)\}_{l=1}^m \textnormal{ is an antichain.}
\Bigg\}
\end{align*}

\begin{lemma}
Let $V$ be a vector space and 
\[
A_1 \subseteq A_2 \subseteq \cdots \subseteq A_m \subseteq X
\]  
\[
X \supseteq B_1 \supseteq B_2  \supseteq \cdots  \supseteq B_m
\]
a pair of semi-flags in $V$, with $m\ge 2$. Then we can interchange unions and intersections in the following shifted manner:
\begin{equation}
	\sum_{i=1}^m A_i \cap B_i = B_1 \cap \left( \bigcap_{i=1}^{m-1} (A_i  + B_{i+1}) \right) \cap A_m
\end{equation}
\end{lemma}

\begin{proof} For $m=2$ the equality reads
\[
A_1\cap B_1 + A_2\cap B_2 = B_1\cap\left(A_1+B_2\right)\cap A_2
\]
Both $A_1\cap B_1$ and $A_2\cap B_2$ are contained in $B_1\cap\left(A_1+B_2\right)\cap A_2$, and so their sum is as well. For inclusion in the other direction, we note that an element $v$ of the 3-fold intersection appearing on the right can be written in three different ways:
\[
v = b_1 = a_1+b_2 = a_2
\]
for elements $a_i\in A_i, b_i\in B_i$. Then $(b_1-b_2) = a_1\in A_1\Rightarrow (b_1-b_2)\in A_1\cap B_1$ and $b_2 = (a_2-a_1)\in A_2\Rightarrow b_2\in A_2\cap B_2$, yielding $v = (b_1-b_2) + b_2\in A_1\cap B_1 + A_2\cap B_2$.
\vskip.1in

Assume $m>2$. By induction on $m$ we have
\[
\sum_{i=1}^m A_i\cap B_i = \left(\sum_{i=1}^{m-1} A_i\cap B_i\right) + A_m\cap B_m
= B_1\cap\left(\bigcap_{i=1}^{m-2} (A_i+B_{i+1})\right)\cap A_{m-1} + A_m\cap B_m
\]
Set $W_1 := B_1\cap\left(\bigcap_{i=1}^{m-2} (A_i+B_{i+1})\right)\cap A_{m-1}$, $W_2 := A_m\cap B_m$, and 
$W_3 := B_1 \cap \left( \bigcap_{i=1}^{m-1} (A_i  + B_{i+1}) \right) \cap A_m$. There are evident inclusions $W_1\subseteq W_3, W_2\subseteq W_3$, implying $W_1 + W_2\subseteq W_3$. For the other direction, we see (as above) that an element $v\in W_3$ can be expressed in $(m+1)$ different ways
\begin{equation}\label{eqn:11}
W_3\ni v = b_1 = a_1+b_2 = a_2+ b_3 =\dots = a_{m-1}+b_m = a_m,\quad a_i\in A_i, b_j\in B_j
\end{equation}
Let $v_1 = b_1 - b_m$. From (\ref{eqn:11}) we have
\[
v_1 = (b_1-b_m) = a_1 + (b_2-b_m) = a_2 + (b_3-b_m) =\dots = a_{m-1} \Rightarrow v_1\in B_1\cap\left(\bigcap_{i=1}^{m-2} (A_i+B_{i+1})\right)\cap A_{m-1}
\]
and 
\[
b_m = a_m-a_{m-1}\in A_m \Rightarrow b_m\in A_m\cap B_m
\]
yielding $v = v_1 + b_m\in W_1 + W_2$.
\end{proof}

\begin{lemma}
	${\cal S}(k) = {\cal S}_r(k) = {\cal I}_r(k) = {\cal I}(k)$
\end{lemma}
\begin{proof}
The middle equality follows by applying the previous lemma. 
One inclusion of the two outer equalities is obvious. The other follows by reducing non-antichains to antichains.
\end{proof}

\begin{proposition}
\label{prop:stable-local-struc-of-pers}
	For all $k \in [1,n]$, the local structure of $M$ stabilizes after $1$ stage, and is ${\cal I}(k) = {\cal S}(k)$. i.e.
	$${\cal F}_1(M)(k) = {\cal I}(k) = {\cal F}(M)(k)$$
	Hence $e(M)=0$.
\end{proposition}
\begin{proof}
Since $
	{\cal F}_0(M)(k) = 
	\left\{
		\im(i,k) \cap \ker(k,j) \ \vrule \
		0 \leq i \leq k \leq j \leq n+1
	\right\},
	$ the first equality holds by applying LS1. The second equality holds since ${\cal I}(k) = {\cal S}(k)$ is biclosed.
	By Lemma \ref{lemma:2} the multiflag generated by the semiflags of images and of kernels is in general position, and therefore so is its sum-intersection closure (Lemma \ref{lemma:generalposition}). The multiflag ${\cal I}(k)$ is a subflag of the sum-intersection closure, hence is in general position for each $k$. We conclude that $e(M) = 0$.
\end{proof}

While $0$ and $n+1$ are not objects in \ul{$n$}, we unify our notation by taking $M(0)=M(n+1)=\{\bf 0\}$ and  $\phi_{0,1}:{\{\bf 0\}} \lra{}{V_1}$ and $\phi_{n,n+1}:V_n \lra{}\{\bf 0\}$ to be zero maps. Note that $\ker(k,k) = 0$ and $\ker(k,n+1) = M(k)$. Let $\im(i,k) = \im(\phi_{i,k})$. Then $\im(k,k) = M(k)$ and $\im(0,k) = \{\bf 0\}$.
\vskip.1in

For $k \in [i,j]$, define the subquotient
\addtolength{\jot}{1em}
\begin{align*}
	\bar A_k &= \frac{N_k}{D_k} = \frac{\ker(k,j+1) \cap \im(i,k)}{\ker(k,j) \cap \im(i,k) \ + \ \im(i-1,k)\cap \ker(k,j+1)}
\end{align*}

\begin{lemma}
For $i\le l < j$, $\bar \phi_{l,l+1}:\bar A_l \lra{} \bar A_{l+1}$ is an isomorphism, and $\bar \phi_{j,j+1}(\bar A_j) = 0$
\end{lemma}
\begin{proof}
 If $i\le l < j$, $\phi_{l,l+1}$ maps $N_l$ surjectively to  $N_{l+1}$, implying $\bar \phi_{l,l+1}$ is surjective. In addition, $ker\left(\phi_{l,l+1}|_{N_k}\right) = ker(k,k+1)\cap im(i,k)\subset D_k$, so $\bar \phi_{l,l+1}$ is injective. Finally $\phi_{j,j+1}$ maps $N_j$ to zero, implying the second statement.
\end{proof}

As $M$ is equipped with an inner product, $A_k := (D_k \subset N_k)^\perp \cong \bar A_k$, and so
\begin{equation}\label{eqn:12}
M[i,j] := 0 \lra{} \cdots \lra{} 0 \lra{} A_i\xrightarrow[\cong]{\phi_{i,i+1}}\cdots \xrightarrow[\cong]{\phi_{j-1,j}} A_j \lra{} 0 \lra{}\cdots\lra{} 0
\end{equation}
is a submodule of $M$ and also a block supported on $[i,j]$.

\begin{theorem}\label{thm:decomp} Let $M : \underline{n} \lra{} (vect/k)$ be a finite persistence module. Then $M$ decomposes uniquely (up to reordering) as a direct sum of blocks:
\[
M=\bigoplus_{[i,j]\subseteq [1,n]} M[i,j]
\]
\end{theorem}

\begin{proof} It is clear that the support of any block must be an interval submodule of $[1,n]$, and moreover that any interval submodule can occur as the support of a block in $M$. By Proposition \ref{prop:stable-local-struc-of-pers}, $M$ is an $\cal IPC$-module with no object-excess; it follows by the results of the previous section and Lemma \ref{lemma:10} that $M$ is the direct sum of its blocks. It is straightforward to show that for each $1\le i\le j\le n$ the $[i,j]$-block appearing in the associated graded of ${\cal F}(M) = {\cal F}_1(M)$ is precisely $M[i,j]$, as defined in (\ref{eqn:12}). The decomposition is unique up to reordering, as it is completely determined by the local structure on $M$.
\end{proof}

It is additionally worth noting that this decomposition is \emph{basis-free}, arising directly from the local structure. To get the traditional decomposition of $M$ as a direct sum of interval modules, one needs to further express each $M[i,j]$ as a direct sum of interval modules of type $[i,j]$. This is a simple exercise which however does require a choice of basis (although the number of terms in the direct sum does not).

\vskip.3in

\subsection{Modules with stable local structure} One-dimensional persistence modules generalize naturally to higher dimensions.

\begin{definition} A \emph{finite $n$-dimensional persistence category $\cal C$} is a finite category $\cal C$ for which there exists an isomorphism of categories
\[
{\rm \underline{m_1}}\times{\rm \underline{m_2}}\times\dots{\rm \underline{m_n}}\cong{\cal C}
\]
for some set of positive integers $m_1,m_2,\dots,m_n$. An $n$-dimensional persistence category $\cal C$ has \emph{multi-dimension} $(m_1,m_2,\dots,m_n)$ if $\cal C$ is isomorphic to $\underline{m_1}\times\underline{m_2}\times\dots\times\underline{m_n}$; note that this $n$-tuple is a well-defined invariant of the isomorphism class of $\cal C$ up to reordering. An \emph{$n$-dimensional persistence module} is a $\cal C$-module for some finite $n$-dimensional persistence category $\cal C$.
\end{definition}

The proof of the next theorem illustrates the usefulness of strong stability.

\begin{theorem}\label{thm:6} Finite $n$-dimensional persistence modules have strongly stable local structure for all $n\ge 0$.
\end{theorem}

\begin{proof} We may assume, without loss of generaality, that $M$ is a $\cal C$-module for an $n$-dimensional persistence category $\cal C$ with multi-dimension $(m_1,m_2,\dots,m_n)$ where the dimensions $m_i$ have been arranged in non-increasing order. We assume the objects of ${\cal C}$ have been labeled with multi-indices $(i_1,i_2,\dots,i_n), 1\le i_j\le m_j$, so that a morphism  exists in ${\cal C}$ from $(i_1,i_2,\dots,i_n)$ to $(j_1,j_2,\dots,j_n)$  iff $i_k\le j_k, 1\le k\le n$ (as $\cal C$ is a poset category, if such a morphism exists it is unique). We will reference the objects of $\cal C$ by their multi-indices. The proof is by induction on dimension; the base case $n=0$ is trivially true as there is nothing to prove.
\vskip.1in
Assume then that $n\ge 1$. For $1\le i\le j\le m_n$, let ${\cal C}[i,j]$ denote the full subcategory of $\cal C$ on objects $(k_1,k_2,\dots,k_n)$ with $i\le k_n\le j$, and let $M[i,j]$ denote the restriction of $M$ to ${\cal C}[i,j]$. Let ${\cal F}_1$ resp.~${\cal F}_2$ denote the local structures on $M[1,m_n-1]$ and $M[m_n]$ respectively; by induction on the cardinality of $m_n$ we may assume these local structures are stable with stabilization indices $N_1,N_2$. Let $\phi_i:M[i]\to m[i+1]$ be the structure map from level $i$ to level $(i+1)$ in the $n$th coordinate. Then define $\phi_\bullet : M[1,m_n-1]\to M[m_n]$ be the morphism of $n$-dimensional persistence modules which on $M[i]$ is given by the composition
\[
M[i]\xrightarrow{\phi_i} M[i+1]\xrightarrow{\phi_{i+1}}\dots M[m_n-1]\xrightarrow{\phi_{m_n-1}} M[m_n]
\]
Define a multi-flag on $M[1,m_n-1]$ by ${\cal F}_1^* := \phi_\bullet^{-1}[{\cal F}_2]$ and on $M[m_n]$ by ${\cal F}_2^* := \phi_\bullet ({\cal F}_1)$. By induction on length and dimension we may assume that $M[1,m_n-1]$ and $M[m_n]$ have local structures which stabilize strongly (we note that $M[m_n]$ is effectively an $(n-1)$-dimensional persistence module). As these multi-flags are finite, we have that
\begin{itemize}
\item the restricted local structures ${\cal F}_i$ are stable (noted above);
\item the local structure of $M[1,m_n-1]$ is stable relative to ${\cal F}_1^*$;
\item the local structure of $M[m_n]$ is stable relative to ${\cal F}_2^*$.
\end{itemize}
We may then choose $N$ so that in each of the three itemized cases, stabilization has been achieved by the $N^{th}$ stage. Let $\cal G$ be the multi-flag on $M$ which on $M[1,m_n-1]$ is the local structure relative to ${\cal F}_1^*$ and on $M[m_n]$ is the local structure relative to ${\cal F}_2^*$. Then $\cal G$ is the local structure on $M$, and has been achieved after at most $2N$ stages starting with the trivial semi-flag on $M$. This implies $M$ has stable local structure. To verify the induction step for the statement that $M$ has srongly stable local structure, let $F$ be a finite multi-flag on $M$. Let $F_1$ be its restriction to $M[1,m_n-1]$, and $F_2$ its restriction to $M[m_n]$. Then let ${\cal F}_i^{**}$ denote the multi-flag generated by ${\cal F}_i^*$ and $F_i$. Proceeding with the same argument as before yields a multi-flag ${\cal G}^*$ achieved at some finite stage which represents the local structure of $M$ relative to $F$, completing the induction step for persistence modules.
\end{proof}
\vskip.2in

The above discussion is independent of the base field; in this completely general case it is possible that the local structure of an arbitrary $\cal C$-module fails to be stable. However, if the base field $k$ is finite, and $\cal C$ is a finite category, then the finiteness of $\cal C$ together with the finite dimensionality of a $\cal C$-module $M$ at each vertex implies that any $\cal C$-module $M$ over $k$ is a finite set. In this case, the infinite refinement of ${\cal F}(M)$ that must occur in order to prevent stabilization at some finite stage is no longer possible. Hence

\begin{theorem} Assume the base field $k$ is finite. Then for all finite categories $\cal C$ and $\cal C$-modules $M$, $M$ has stable local structure.
\end{theorem}
\vskip.5in

%%%%%%%%%%%%%%%%%%%%%%%%%%%%%%%%%%%%%%%%%%%%%%%%%
%%%%%%%%%%%%%%%%%%%%%%%%%%%%%%%%%%%%%%%%%%%%%%%%%

\section{Topologically based $\cal C$-modules} A $\cal C$-module $M$ is said to be  {\it topologically based} if $M = H_*(F;k)$ (for either a particular value of $*$ in the ungraded case, or more generally viewed as a graded element of $(vect/k)$), where $F:{\cal C}\to {\cal D}$ is a functor from $\cal C$ to a category $\cal D$, equalling either
\begin{itemize}
\item {\bf f-s-sets} - the category of simplicial sets with finite skeleta and morphisms of simplicial sets, or
\item {\bf f-s-com} - the category of finite simplicial complexes and morphisms of simplicial complexes.
\end{itemize}

In what follows we will restrict ourselves to the category {\bf f-s-sets}, as it is slightly easier to work in (although all results carry over to {\bf f-s-complexes}). We show that any topologically-based module indexed on a poset category $\cal C$ admits a presentation by $\cal IPC$-modules (hence a presentation by $\cal FIPC$-modules by our results above). We also prove a general K\"unneth Theorem for topologically based $\cal C$-modules without restrictions on $\cal C$ (answering a question posed by G.~Carlsson).
\vskip.2in

\subsection{An ${\cal IPC}$-presentation} For this subsection we assume $\cal C$ to be a connected, finite poset-category, so that all $\cal C$-modules are finite poset-modules. 
\vskip.2in

We first show that any $\cal C$-diagram in {\bf f-s-sets} can be cofibrantly replaced, up to weak homotopical transformation. Precisely,

\begin{theorem} If $F:{\cal C}\to$ {\bf f-s-sets}, then there is a $\cal C$-diagram $\wt{F}:{\cal C}\to$ {\bf f-s-sets} and a natural transformation $\eta:\wt{F}\xrightarrow{\simeq} F$ which is a weak equivalence at each object, where $\wt{F}(\phi_{xy})$ is a closed cofibration (inclusion of simplicial sets) for all morphisms $\phi_{xy}$\footnote{The proof following is a minor elaboration of an argument communicated to us by Bill Dwyer \cite{bd}.}.
\end{theorem}

\begin{proof} The simplicial mapping cylinder construction $Cyl(_-)$ applied to any morphism in {\bf f-s-sets} verifies the statement of the theorem in the simplest case $\cal C$ consists of two objects and one non-identity morphism. Suppose $\cal C$ has $n$ objects; we fix a total ordering on $obj({\cal C})$ that refines the partial ordering: $\{x_1 \prec x_2 \prec \dots \prec x_n\}$ where if $\phi_{x_i x_j}$ is a morphism in $\cal C$ then $i\le j$ (but not necessarily conversely). Let ${\cal C}(m)$ denote the full subcategory of $\cal C$ on objects $x_1,\dots,x_m$, with $F_m = F|_{{\cal C}(m)}$. By induction, we may assume the statement of the theorem for $F_m:{\cal C}(m)\to$ {\bf f-s-sets}, with cofibrant lift denoted by $\wt{F}_m$; with $\eta_m:\wt{F}_m\xrightarrow{\simeq} F_m$.
\vskip.2in

Now let ${\cal D}(m)$ denote the slice category ${\cal C}/x_{m+1}$; as ``$\prec$" is a refinement of the poset ordering ``$<$", the image of the forgetful functor $P_m:{\cal D}(m)\to {\cal C}; (y\to x_{m+1})\mapsto y$ lies in ${\cal C}(m)$. And as $\cal C$ is a poset category, the collection of morphisms $\{\phi_{y x_{m+1}}\}$ uniquely determine a map
\[
f_m : \underset{{\cal D}(m)}{colim}\ \wt{F}_m\circ P_m\xrightarrow{\eta_m} \underset{{\cal D}(m)}{colim}\  F_m\circ P_m \to F(x_{m+1})
\]
Define $\wt{F}_{m+1}:{\cal C}(m+1)\to$ {\bf f-s-sets} by
\begin{itemize}
\item $\wt{F}_{m+1}|_{{\cal C}(m)} = \wt{F}_m$;
\item $\wt{F}_{m+1}(x_{m+1}) = Cyl(f_m)$;
\item If $\phi_{x x_{m+1}}$ is a morphism from $x\in obj({\cal C}(m))$ to $x_{m+1}$, then
\[
\wt{F}_{m+1}(\phi_{x x_{m+1}}):\wt{F}_{m}(x) = \wt{F}_{m+1}(x)\to \wt{F}_{m+1}(x_{m+1})
\]
 is given as the composition
 \[
 \wt{F}_{m}(x) = \wt{F}_m\circ P_m(x\xrightarrow{\phi_{x x_{m+1}}} x_{m+1})\hookrightarrow 
 \underset{{\cal D}(m)}{colim}\ \wt{F}_m\circ P_m\hookrightarrow Cyl(f_m) = \wt{F}_{m+1}(x_{m+1})
 \]
\end{itemize}
where the first inclusion into the colimit over ${\cal D}(m)$ is induced by the inclusion of the object \newline
$(x\xrightarrow{\phi_{x x_{m+1}}} x_{m+1})\hookrightarrow obj({\cal D}(m))$. As all morphisms in ${\cal D}(m)$ map to simplicial inclusions under $\wt{F}_m\circ P_m$ the resulting map of $\wt{F}_m(x)$ into the colimit will also be a simplicial inclusion. Finally, the natural transformation $\eta_m:\wt{F}_m\to F_m$ is extended to $\eta_{m+1}$ on $\wt{F}_{m+1}$ by defining $\eta_{m+1}(x_{m+1}): \wt{F}_{m+1}(x_{m+1})\to F_{m+1}(x_{m+1})$ as the natural collapsing map $Cyl(f_m)\surj F(x_{m+1})$, which has the effect of making the diagram

\centerline{
\xymatrix{
\wt{F}_{m+1}(x)\ar[rr]^{\wt{F}_{m+1}(\phi_{xy})}\ar[dd]^{\eta_{m+1}(x)} && \wt{F}_{m+1}(y)\ar[dd]^{\eta_{m+1}(y)}\\
\\
F_{m+1}(x)\ar[rr]^{F_{m+1}(\phi_{xy})} && F_{m+1}(y)
}}
\vskip.2in

commute for morphisms $\phi_{xy}\in Hom({\cal C}_{m+1})$. This completes the induction step, and the proof.
\end{proof}

\begin{corollary}\label{cor:pres} Any topologically based $\cal C$-module $M$ admits a presentation by $\cal C$-modules $N_1\inj N_2\surj M$ where $N_i$ is an $\cal IPC$-module and $N_1\inj N_2$ is an isometric inclusion of $\cal IPC$-modules.
\end{corollary}

\begin{proof} By the previous result and the homotopy invariance of homology, we may assume $M = H_n(F)$ where $F :{\cal C}\to$ {\bf i-f-s-sets}, the subcategory of {\bf f-s-sets} on the same set of objects, but where all morphisms are simplicial set injections. In this case, for each object $x$, $C_n(F(x)) = C_n(F(x);k)$ admits a canonical inner product determined by the natural basis of $n$-simplices $F(x)_n$, and each morphism $\phi_{xy}$ induces an injection of basis sets $F(x)_n\inj F(y)_n$, resulting in an isometric inclusion $C_n(F(x))\inj C_n(F(y))$. In this way the functor $C_n(F) := C_n(F;k):{\cal C}\to (vect/k)$ inherits a natural $\cal IPC$-module structure. If $Q$ is an $\cal IPC$-module where all of the morphisms are isometric injections, then any $\cal C$-submodule $Q'\subset Q$, equipped with the same inner product, is an $\cal IPC$-submodule of $Q$. Now $C_n(F)$ contains the $\cal C$-submodules $Z_n(F)$ ($n$-cycles) and $B_n(F)$ ($n$-boundaries); equipped with the induced inner product the inclusion $B_n(F)\hookrightarrow Z_n(F)$ is an isometric inclusion of $\cal IPC$-modules, for which $M$ is the cokernel $\cal C$-module. 
\end{proof}

[Note: The results for this subsection have been stated for {\bf f-s-sets}; similar results can be shown for {\bf f-s-complexes} after fixing a systematic way for representing the mapping cyclinder of a map of simplicial complexes as a simplicial complex; this typically involves barycentrically subdividing.]

%%%%%%%%%%%%%%%%%%%%%%%%%%%%%%%%%%%%%%%%%%%%%%%%%

\subsection{A K\"unneth Theorem for topologically based $\cal C$-modules} The following theorem represents the natural generalization of the traditional K\"unneth Theorem for singular homology with field coefficients.

\begin{theorem} Let $F_i : {\cal C}_i\to \text{\bf f-s-sets}, 1\le i\le n$, with $M_i := H_*(F_i;k)$ denoting the corresponding functor from ${\cal C}_i$ to the category $(gr\mhyphen vect/k)$ of finite-dimensional graded vector spaces over $k$. Similarly, define
\begin{gather*}
F := diag(F_1\times F_2\times\dots F_n):{\cal C}_1\times{\cal C}_2\times\dots {\cal C}_n\to \text{\bf f-s-sets}\\
M := H_*(F;k)
\end{gather*}
Then there is a natural isomorphism
\begin{equation}\label{eqn:kt1}
M\cong M_1\otimes M_2\otimes\dots\otimes M_n
\end{equation}
of topological ${\cal C}_1\times{\cal C}_2\times\dots {\cal C}_n$-modules which induces isomorphisms
\begin{equation}\label{eqn:kt2}
{\cal F}(M)\cong {\cal F}(M_1)\otimes{\cal F}(M_2)\otimes\dots {\cal F}(M_n)
\end{equation}
\begin{equation}\label{eqn:kt3}
{\cal F}(M)_*\cong {\cal F}(M_1)_*\otimes{\cal F}(M_2)_*\otimes\dots {\cal F}(M_n)_*
\end{equation}
\end{theorem}

\begin{proof} The proof is short, as the main technical component has already been established above. The equality in (\ref{eqn:kt1}) follows from the traditional K\"unneth Theorem. Applying Theorem \ref{thm:kunneth} to the RHS of (\ref{eqn:kt1}) then yields the isomorphisms of (\ref{eqn:kt2}) and (\ref{eqn:kt3}). We note that as $M$ is graded, the tensor products in all three equations are to be taken in the graded sense.
\end{proof}

\begin{corollary} If $M_i, 1\le i\le n$ (as in the previous theorem) are tame for $1\le i\le n$, then so is $M_1\otimes M_2\otimes\dots M_n$.
\end{corollary}

This applies in particular to the case that each $M_i$ is a finite 1-dimensional persistence module.
\vskip.3in

\section*{Funding and Competing Interests} No funding was received to assist with the preparation of this manuscript.  The authors have no competing interests to declare that are relevant to the content of this article.

\section*{Data Availability Statement} Data sharing is not applicable to this article as no new data were created or analyzed in this study.

\end{document}